\documentclass[12pt]{amsart}
\usepackage[utf8]{inputenc}
\usepackage[shortlabels]{enumitem}
\usepackage{amssymb}
\usepackage{mathrsfs}
\usepackage{graphicx}
\usepackage{stmaryrd}
\usepackage[all]{xy}
\usepackage{xcolor}
\usepackage[margin=1in]{geometry} 
\usepackage[bookmarks, bookmarksdepth=2, colorlinks=true, linkcolor=blue, citecolor=blue, urlcolor=blue]{hyperref}
\usepackage{eucal}

\numberwithin{equation}{section}
\newtheorem{theorem}[equation]{Theorem}

\newtheorem{proposition}[equation]{Proposition}
\newtheorem{lemma}[equation]{Lemma}
\newtheorem{corollary}[equation]{Corollary}

\theoremstyle{definition}
\newtheorem{rmk}[equation]{Remark}
\newenvironment{remark}[1][]{\begin{rmk}[#1] \pushQED{\qed}}{\popQED \end{rmk}}
\newtheorem{eg}[equation]{Example}
\newenvironment{example}[1][]{\begin{eg}[#1] \pushQED{\qed}}{\popQED \end{eg}}
\newtheorem{defnaux}[equation]{Definition}
\newenvironment{definition}[1][]{\begin{defnaux}[#1]\pushQED{\qed}}{\popQED \end{defnaux}}

\newcommand{\bA}{\mathbf{A}}

\newcommand{\bC}{\mathbf{C}}

\newcommand{\bF}{\mathbf{F}}

\newcommand{\bG}{\mathbf{G}}

\newcommand{\bN}{\mathbf{N}}

\newcommand{\bP}{\mathbf{P}}
\newcommand{\cP}{\mathcal{P}}

\newcommand{\bQ}{\mathbf{Q}}

\newcommand{\bZ}{\mathbf{Z}}

\newcommand{\bm}{\mathbf{m}}

\newcommand{\arxiv}[1]{\href{http://arxiv.org/abs/#1}{{\tiny\tt arXiv:#1}}}
\newcommand{\DOI}[1]{\href{http://doi.org/#1}{\color{purple}{\tiny\tt DOI:#1}}}

\newcommand{\defn}[1]{\emph{#1}}
\let\ol\overline
\let\ul\underline
\renewcommand{\phi}{\varphi}

\DeclareMathOperator{\Sym}{Sym}
\DeclareMathOperator{\str}{str}
\DeclareMathOperator{\astr}{astr}
\DeclareMathOperator{\len}{len}

\newcommand{\lpp}{(\!(}
\newcommand{\rpp}{)\!)}

\makeatletter
\newcounter{ccount}
\newcommand{\newC}[1]{\stepcounter{ccount}\expandafter\edef\csname #1\endcsname{\theccount}}
\newcommand{\C}[1]{C_{\csname #1\endcsname}}
\makeatother

\title{Two improvements in Brauer's theorem on forms}

\author{Arthur Bik}
\address{University of Neuch\^atel, Switzerland, and Institute for Advanced Study, Princeton, NJ, USA}
\email{\href{mailto:mabik@ias.edu}{mabik@ias.edu}}
\urladdr{\url{http://arthurbik.nl}}

\author{Jan Draisma}
\address{University of Bern, Switzerland, and Eindhoven University of Technology, The Netherlands}
\email{\href{mailto:jan.draisma@math.unibe.ch}{jan.draisma@math.unibe.ch}}
\urladdr{\url{https://mathsites.unibe.ch/jdraisma/}}

\author{Andrew Snowden}
\address{Department of Mathematics, University of Michigan, Ann Arbor, MI, USA}
\email{\href{mailto:asnowden@umich.edu}{asnowden@umich.edu}}
\urladdr{\url{http://www-personal.umich.edu/~asnowden/}}

\thanks{AB was partially supported by Postdoc.Mobility Fellowship number P400P2\_199196 from the Swiss National Science
Foundation and a grant from the Simons Foundation (816048, LC).
JD was partially supported by the Vici grant 639.033.514 from the Netherlands Organisation for Scientific Research and by Swiss National Science Foundation project grant 200021\_191981.
AS was supported by NSF grant DMS-2301871.}

\date{January 3, 2024}

\begin{document}

\begin{abstract}
Let $k$ be a \defn{Brauer field}, that is, a field over which every diagonal form in sufficiently many variables has a nonzero solution; for instance, $k$ could be an imaginary quadratic number field. Brauer proved that if $f_1, \ldots, f_r$ are homogeneous polynomials on a $k$-vector space $V$ of degrees $d_1, \ldots, d_r$, then the variety $Z$ defined by the $f_i$'s has a non-trivial $k$-point, provided that $\dim{V}$ is sufficiently large compared to the $d_i$'s and $k$. We offer two improvements to this theorem, assuming $k$ is infinite. First, we show that the Zariski closure of the set $Z(k)$ of $k$-points has codimension $<C$, where $C$ is a constant depending only on the $d_i$'s and $k$. And second, we show that if the strength of the $f_i$'s is sufficiently large in terms of the $d_i$'s and $k$, then $Z(k)$ is actually Zariski dense in $Z$. The proofs rely on recent work of Ananyan and Hochster on high strength polynomials.
\end{abstract}

\maketitle
\tableofcontents

\section{Introduction}

\subsection{Results}

In 1945, Richard Brauer \cite{Brauer} proved a now-famous theorem on the solutions to a system of homogeneous polynomial equations. The purpose of this paper is to offer two improvements on Brauer's theorem. To discuss this circle of ideas, it will be convenient to introduce the following terminology.

\begin{definition}
A field $k$ is a \defn{Brauer field} if for every $d \ge 1$, there is a quantity $N_k(d)$ such that any equation of the form
\begin{displaymath}
a_1 x_1^d + \ldots + a_n x_n^d = 0
\end{displaymath}
with $n>N_k(d)$ and $a_i \in k$ has a non-trivial solution in $k$.
\end{definition}

Examples of Brauer fields include finite fields, $p$-adic fields, function fields, and non-real number fields; see \S \ref{s:brauer} for details.

Fix a Brauer field $k$. Let $\ul{f}=(f_1, \ldots, f_r)$ be an $r$-tuple
of homogeneous polynomials defined on the finite dimensional $k$-vector
space $V$ of degrees $\ul{d}=(d_1, \ldots, d_r)$, and let $Z=Z(\ul{f})$
be the closed subvariety of $V$
defined by the vanishing of the $f_i$'s. We begin by recalling Brauer's
original result:

\newC{c:brauer}
\begin{theorem}[Brauer] \label{thm:brauer}
If $\dim(V)>\C{c:brauer}(\ul{d})$, then $Z$ has a nonzero $k$-point.
\end{theorem}

This theorem is really asserting that a constant $\C{c:brauer}$ exists
making the statement true. We employ this basic structure throughout
the paper: if a new constant $C_j$ 
appears in a theorem, then the theorem is asserting its existence. The
constant $C_j$ may depend on the field but we never indicate this in the notation.

Our two main theorems improve Brauer's theorem by showing that $Z$ in fact possesses an abudance of $k$-points. In these theorems, we must assume that the field $k$ is infinite. Our first main theorem is the following:

\newC{c:mainthm2}
\begin{theorem} \label{mainthm2}
The Zariski closure of $Z(k)$ in $Z$ has codimension $\le \C{c:mainthm2}(\ul{d})$.
\end{theorem}

If $\dim(V)$ exceeds $\C{c:mainthm2}(\ul{d})+r$ then Theorem~\ref{mainthm2}
implies that the Zariski closure of $Z(k)$ has positive dimension, and
so $Z(k)$ is infinite. In particular, this means that $Z(k)$ contains
a nonzero point, and so Theorem~\ref{mainthm2} recovers
Theorem~\ref{thm:brauer}. 

In their work on Stillman's conjecture, Ananyan and Hochster \cite{AH}
defined a notion of rank for polynomials, or tuples of polynomials,
that they called \defn{strength}\footnote{Schmidt \cite{Schmidt} had
previously defined the same invariant.}. We recall the definition in
\S \ref{ss:strength}, and write $\str(\ul{f})$ for the strength of the
tuple $\ul{f}$. Our second result applies in the high strength regime:

\newC{c:mainthm}
\begin{theorem} \label{mainthm}
If $\str(\ul{f})>\C{c:mainthm}(\ul{d})$ then $Z(k)$ is Zariski dense in $Z$.
\end{theorem}

This theorem rather easily implies Theorem~\ref{mainthm2}; see \S
\ref{ss:pf-mainthm2}. The majority of the paper is thus devoted to
proving Theorem~\ref{mainthm}. We give one simple application of the
theorem here:

\newC{c:maincor}
\begin{corollary}
Let $d \ge 1$ be an integer that is not a multiple of the characteristic of $k$. Then any equation
\begin{displaymath}
a_1 x_1^d + \cdots + a_n x_n^d = 1,
\end{displaymath}
with $a_1, \ldots, a_n \in k$ nonzero and $n \geq \C{c:maincor}(d)$ has a solution in $k$.
\end{corollary}

\begin{proof}
We take $\C{c:maincor}(d)=2\C{c:mainthm}((d))$. Let $f=\sum_{i=1}^n
a_i x_i^d - x_{n+1}^d$. This has strength at least $(n+1)/2$ by
Proposition~\ref{prop:diagstrength}. Since this exceeds
$\C{c:mainthm}(d)$, Theorem~\ref{mainthm} implies that the
$k$-points of $Z(f)$ are Zariski dense. In particular, there is a
$k$-solution to $f=0$ with $x_{n+1}$ non-zero. Dividing by $x_{n+1}^d$
yields the result. 
\end{proof}

\subsection{Idea of proof}

We begin by recalling the basic idea behind Brauer's original proof.
We proceed by induction on the degree tuple $\ul{d}$; here we assume
$d_1 \ge d_2 \ge \cdots \ge d_r$, and we order such tuples
lexicographically. To illustrate the main idea, let us treat the case
$r=1$. Thus we have a single homogeneous polynomial $f$ of degree $d$,
and we are assuming that the result holds for all tuples of polynomials of degrees $<d$.

Pick any nonzero vector $v_1 \in V$. For another vector $v_2$, write
\begin{displaymath}
f(x_1 v_1+x_2 v_2) = \sum_{i+j=d} f_j(v_2) x_1^i x_2^j,
\end{displaymath}
where $f_j$ is a degree-$j$ polynomial in $v_2$; here we are treating $v_1$ as constant and $v_2$ as variable. Note that $f_d=f$. By the inductive hypothesis, we can find $v_2 \in V$ such that $f_j(v_2)=0$ for all $0<j<d$; we can also ensure $v_2$ is not in the span of $v_1$ by first passing to an appropriate complementary subspace. We thus have
\begin{displaymath}
f(x_1 v_1+x_2 v_2) = f(v_1) x_1^d + f(v_2) x_2^d.
\end{displaymath}
We say that $v_1$ and $v_2$ are \defn{$f$-orthogonal}. Continuing in this way, we find a linearly independent $f$-orthogonal sequence $v_1, \ldots, v_n$. We then have
\begin{displaymath}
f(x_1 v_1+\ldots + x_n v_n) = f(v_1) x_1^d + \ldots + f(v_n) x_n^d.
\end{displaymath}
By taking $n>N_k(d)$, we obtain a non-trivial solution by appealing to the defining property of Brauer fields.

Our proof of Theorem~\ref{mainthm} follows a similar plan: we proceed by induction on $\ul{d}$ and use the inductive hypothesis to construct a linear subspace of $V$ on which the given polynomials have a simple form. However, the argument is now somewhat more subtle.

Once again, we illustrate the idea for $r=1$; we also assume $\operatorname{char}(k)=0$ for simplicity. Thus suppose $f$ is a high strength polynomial of degree $d$, and we know Theorem~\ref{mainthm} holds for all tuples of polynomials of degrees $<d$. We aim to show that the $k$-points of $Z=Z(f)$ are dense. For this, it suffices to fix a polynomial $g$ that is not identically zero on $Z$, and construct a $k$-point of $Z$ at which $g$ does not vanish.

If we simply proceed as in the original argument, we run into an
issue: $g$ could potentially vanish on the entire span of $v_1,
\ldots, v_n$. To address this issue, we begin by finding a subspace
$F$ of $V$ of small dimension such that $g$ does not vanish
identically on $Z \cap F$. We then find $v_1, \ldots, v_n$ that are
$f$-orthogonal to each other and to $F$. We now also add the constraint that $f(v_i) \ne 0$ for all $i$; we can ensure this by making use of the full force of our inductive hypothesis. The restriction of $f$ to $\operatorname{span}(v_1, \ldots, v_n) \oplus F$ has the form
\begin{displaymath}
a_1 x_1^d + \ldots + a_n x_n^d + h(w_1, \ldots, w_m),
\end{displaymath}
where the $a_i$ are nonzero and $h$ is some polynomial that we have no information about; the $w_j$ here are the coordinates on $F$.

We next show that we can find a three dimensional subspace $E$ of
$\operatorname{span}(v_1, \ldots, v_n)$ such that the restriction of
$f$ has the form $xy^{d-1}+ay^d+bz^d$, where $a$ and $b$ are scalars
with $b$ nonzero and $x$, $y$, and $z$ are suitable coordinates on $E$. This uses the defining property of Brauer fields, as well as our inductive hypothesis. We thus see that the restriction of $f$ to $W=E \oplus F$ is
\begin{displaymath}
xy^{d-1}+a y^d+b z^d + h(w_1, \ldots, w_m).
\end{displaymath}
Since we can (rationally) solve the above equation for $x$, it follows that $Z \cap W$ is a rational variety; thus its $k$-points are dense (since $k$ is infinite). Since $g$ is not identically~0 on $Z \cap W$ (as it is not identically zero on $Z \cap F$), we can therefore find a $k$-point where $g$ is nonzero.

To carry out the details in the above argument, we must make use of deep properties of high strength polynomials established by Ananyan and Hochster in the course of their proof of Stillman's conjecture \cite{AH}. In positive characteristic the above argument does not quite work, since diagonal forms are not general enough, and a slight modification is required.

\subsection{Motivation}

Our motivation for this work stemmed from a question on the universality of high strength tensors.

Fix an algebraically closed field $k$, an integer $d \ge 1$ such that
$d!$ is invertible in $k$, and $m \ge 1$. Let $f$ be a degree $d$
polynomial on a vector space $V$ with strength sufficiently large
compared to $k$, $d$, and $m$. Given any degree $d$ polynomial $g$ on
a vector space $W$ of dimension $\le m$, there is a linear embedding
$i \colon W \to V$ such that $g=i^*(f)$. Thus $f$ is ``$m$-universal''
in the sense that one can obtain any polynomial in $m$ variables from
$f$ by a linear substitution. This important result is due to Kazhdan
and Ziegler \cite[Theorem~1.12]{KaZ1}. Subsequently, \cite{BDDE} extended this result to arbitrary tensors.

In the above setting, if we regard $f$ and $g$ as given then we can view $g=i^*(f)$ as a system of polynomial equations in the matrix entries of $i$. Since $\dim(V)$ is large and $\dim(W)$ is small, this sytem contains a small number of polynomials in a large number of variables. In view of Brauer's theorem, it is therefore reasonable to ask if this system can be solved over a Brauer field. Brauer's theorem on its own is insufficient to prove this, but the results of this paper are exactly what is needed. In a forthcoming paper \cite{BDS1} we carry out the details, thereby generalizing the universality resluts to the case when $k$ is a Brauer field.

\subsection{Related work}

There are many papers devoted to improving the constant $\C{c:brauer}$ in Brauer's theorem by algebraic methods; see, for example, \cite{Leep, LS, IL, Wooley1, Wooley2}.

Another thread of related work studies solutions to homogeneous equations in many variables over number fields, by analytic methods; see, for example, \cite{Birch1, Birch2, BHB, BHB2, FM, Schmidt, Skinner}.

\subsection{Outline}

In \S \ref{s:brauer} we give some examples and basic properties of Brauer fields. In \S \ref{s:prelim} we review some results on strength. In \S \ref{s:red}, we reduce our main theorems to Proposition~\ref{prop:good}. This proposition is proved in \S \ref{s:normal0} (in characteristic~0) and \S \ref{s:normalp} (in characteristic $p$).

\subsection{Notation and conventions}

We regard the field $k$ as fixed. The constants appearing in various results depend on $k$, but we do not indicate this.

\subsection*{Acknowledgements}

We thank Amichai Lampert for helpful discussions.

\section{Brauer fields} \label{s:brauer}

We collect some examples and basic properties of Brauer fields.

\subsection{Formally real fields}

Recall that the \defn{Stufe} or \defn{level} of a field $k$, denoted $s(k)$, is the minimal $s$ such that $-1$ is a sum of $s$ squares in $k$; if no such $s$ exists then $s(k)=\infty$. The field $k$ is called \defn{non-real} if $s(k)$ is finite, and \defn{formally real} if $s(k)=\infty$. If $k$ is a Brauer field then $x_1^2+\cdots+x_n^2=0$ has a non-trivial solution when $n=1+N_k(2)$, and so $s(k) \le N_k(2)$. In particular, a Brauer field must be non-real.

\subsection{Number fields}

A number field $k$ is formally real if and only if it admits a real embedding. We have seen that such fields are not Brauer. On the other hand:

\begin{proposition}[{Peck, \cite{Peck}}]
A non-real number field is a Brauer field.
\end{proposition}

We note that the proof of Peck's theorem is analytic: it uses the circle method. Every other result discussed in this section is purely algebraic.

\subsection{Few power classes}

We now investigate fields with few power classes.

\begin{proposition}
Let $k$ be a non-real field such that $k^{\times}/(k^{\times})^d$ is finite for all $d \ge 1$. Then $k$ is a Brauer field.
\end{proposition}

\begin{proof}
Let $d$ be given. Let $n$ be the cardinality of $k^{\times}/(k^{\times})^d$, and let $b_1, \ldots, b_n \in k^{\times}$ be representatives for the quotient. Write $-1=c_1^d+\cdots+c_m^d$, where $c_1, \ldots, c_m \in k$. If $k$ has positive characteristic $p$ then it contains $\bF_p$, and it is easy to see that one has such an expression in $\bF_p$. If $k$ has characteristic~0 then the existence of such an expression is a non-trivial result of Joly \cite[Theorem~6.15]{Joly}.

We claim that $N_k(d) \le nm$. Indeed, consider an equation
\begin{displaymath}
a_1 x_1^d + \cdots + a_{nm+1} x_{nm+1}^d = 0.
\end{displaymath}
If any $a_i$ vanishes then we can simply take $x_i=1$ and $x_j=0$ for $j \ne i$. Thus assume all $a_i$'s are non-zero. Each one belongs to one of the $n$ classes of $d$th powers, and so one of these classes occurs at least $m+1$ times; reordering if necessary, assume that $a_1, \ldots, a_{m+1}$ all belong to the same class, say that of $b_s$. Writing $a_i=b_s (a_i')^d$ for $1 \le i \le m+1$, we have
\begin{displaymath}
a_1 x_1^d+ \cdots + a_{m+1} x_{m+1}^d = b_s( (a_1' x_1)^d + \cdots + (a_{m+1}' x_{m+1})^d).
\end{displaymath}
We now obtain a non-trivial solution to the original equation by taking $x_i=c_i/a_i'$ for $1 \le i \le m$ and $x_{m+1}=1/a_{m+1}'$ and $x_i=0$ for $i>m+1$.
\end{proof}

\begin{corollary}
The following fields are Brauer:
\begin{enumerate}
\item any finite field,
\item any finite extension of $\bQ_p$, and 
\item the field $\bC\lpp t \rpp$ of complex Laurent series.
\end{enumerate}
\end{corollary}

\begin{proposition} \label{prop:ext-few}
Let $k$ be a Brauer field, and let $K/k$ be a field extension. Suppose that $K^{\times}/(k^{\times} \cdot (K^{\times})^d)$ is finite for all $d \ge 1$. Then $K$ is a Brauer field.
\end{proposition}

\begin{proof}
Let $d$ be given. Let $n$ be the cardinality of $K^{\times}/(k^{\times} \cdot (K^{\times})^d)$, and let $b_1, \ldots, b_n \in K^{\times}$ be representatives for the quotient. Let $m=N_k(d)$. We claim that $N_K(d) \le nm$. Indeed, consider an equation
\begin{displaymath}
a_1 x_1^d+\cdots+a_{nm+1} x_{nm+1}^d =0.
\end{displaymath}
If some $a_i$ vanishes then we have a non-trivial solution, so we can assume each $a_i$ is non-zero. Reordering if necessary, we can suppose that $a_1, \ldots, a_{m+1}$ all belong to the class of some $b_s$. For $1 \le i \le m+1$, write $a_i=b_s a_i' (a_i'')^d$ with $a'_i \in k$ and $a''_i \in K$. By assumption, we can find $y_1, \ldots, y_{m+1} \in k$, not all zero, such that $a_1' y_1^d+\cdots+a_{m+1}' y_{m+1}^d=0$. We thus obtain a solution to the original equation by taking $x_i=y_i/a_i''$ for $1 \le i \le m+1$ and $y_i=0$ for $i>m+1$.
\end{proof}

\begin{corollary}
If $k$ is a Brauer field of characteristic~0 then the field of Laurent series $k\lpp t \rpp$ is a Brauer field.
\end{corollary}

\begin{proof}
Suppose $f(t)$ is a non-zero Laruent series with coefficients in $k$. We can then write $f(t)= c t^n g(t)$ where $c \in k$ and $g(t)$ is a power series with constant coefficient~1. Such a $g(t)$ is a $d$th power, for any $d \ge 1$. Thus if $K=k\lpp t \rpp$ then we see that $K^{\times}/(k^{\times} (K^{\times})^d)$ has cardinality $d$, with the classes represented by $1, t, \ldots, t^{d-1}$.
\end{proof}

\subsection{\texorpdfstring{$(C_r)$}{C\_r} fields}

Following Lang \cite{Lang}, we say that a field $k$ satisfies property \defn{$(C_r)$} if for any $d \ge 1$,
any homogeneous polynomial of degree $d$ in $>d^r$ variables with coefficients in $k$ has a
nonzero solution in $k$. It is clear that a $(C_r)$ field is a Brauer field.

A field is $(C_0)$ if and only if it is algebraically closed. The $(C_1)$ condition is also known as quasi-algebraically closed; some examples:
\begin{enumerate}
\item Any finite field is $(C_1)$; this is the Chevalley--Warning theorem \cite[Theorem~2.3]{Greenberg}.
\item If $k$ is algebraically closed and $K/k$ is finitely generated of transcendence degree one then $K$ is $(C_1)$; this is Tsen's theorem \cite[Theorem~1.2]{Greenberg}.
\item If $E/\bQ_p$ is a finite extension then the maximal unramified extension of $E$, denoted $E^{\rm un}$, is $(C_1)$ \cite[Theorem~12]{Lang}.
\item The subfield of $\ol{\bQ}_p\lpp t \rpp$ consisting of convergent series is $(C_1)$ \cite[Theorem~12]{Lang}.
\end{enumerate}
Artin conjectured that $\bQ_p$ is $(C_2)$, but this was disproved by Terjanian \cite[Theorem~7.4]{Greenberg}; the Ax--Kochen theorem provides a positive result in this direction \cite[Theorem~7.5]{Greenberg}. A theorem of Lang--Nagata asserts that if $k$ is a $(C_r)$ field and $k'/k$ is a finitely generated extension of transcendence degree $s$ then $k'$ is a $(C_{r+s})$ field
\cite[Theorem~3.4]{Greenberg}. It is also true that if $k$ is $(C_r)$ then the field $k\lpp t \rpp$ of Laurent series is $(C_{r+1})$ \cite[Theorem~4.8]{Greenberg}. See \cite[Ch.~21]{FriedJarden} for additional details.

We extract a few concrete examples from the above discussion:

\begin{proposition}
Let $k$ be a finite field, $E^{\rm un}$ (for $E/\bQ_p$ finite), or an algebraically closed field. Then:
\begin{enumerate}
\item Any finitely generated extension of $k$ is a Brauer field.
\item The field of Laurent series $k\lpp t \rpp$ is a Brauer field.
\end{enumerate}
\end{proposition}

\subsection{Extension fields}

It is natural to ask how the Brauer property behaves in field extensions. We have already seen one such result (Proposition~\ref{prop:ext-few}). We now give another:

\begin{proposition}
A finite extension $K$ of a Brauer field $k$ is a Brauer field.
\end{proposition}

\begin{proof}
Let $m=[K:k]$ and let $b_1, \ldots, b_m$ be a $k$-basis for $K$. Let $f$ be a diagonal
form over $K$ of degree $d$ in the variables $x_1,\ldots,x_n$. Expressing the coefficients of $f$ in our basis for $K$ yields an expression
\[
f=\sum_{i=1}^m f_i(x_1,\ldots,x_n)\cdot b_i
\]
where the $f_i$ are homogeneous forms of degree $d$ over $k$ (indeed,
diagonal ones, but we will not need this). If $n$ exceeds the constant
$\C{c:brauer}((d,\ldots,d))$ from Brauer's Theorem~\ref{thm:brauer}, where
the number of copies of $d$ is $m$, then the equations $f_i=0$
have a common nontrivial solution over $k$. This is \emph{a fortiori}
a nontrivial solution of $f=0$ over $K$.
\end{proof}

In view of the above proposition, one might ask if $k(t)$ is a Brauer field when $k$ is. We do not know the answer, even in some basic cases, e.g., if $k$ is a non-real number field or $p$-adic field. Positive results in this direction are likely to be quite difficult, since analogous questions about quadratic forms are already difficult or open. Recall that the \emph{$u$-invariant} of a field $k$, denoted $u(k)$, is the maximal dimension of an anisotropic quadratic form over $k$. Brauer's theorem implies that a Brauer field has finite $u$-invariant. (In characteristic $\ne 2$, this follows directly from the definition of Brauer field and the existence of orthogonal bases.) It is at present an open problem if $u(k(t))$ is finite when $k$ is a non-real number field, though it is expected to be~8. It was only recently shown that $u(k(t))=8$ when $k$ is a $p$-adic field with $p \ne 2$ \cite{PS}. See \cite[Ch.~XIII, \S 6]{Lam} for more details.

\subsection{Semi-perfect fields}

Recall that the field $k$ is \defn{perfect} if it has
characteristic~0, or it has characteristic $p>0$ and $k=k^p$. We say
that $k$ is \defn{semi-perfect} if it has characteristic~0, or it has
characteristic $p>0$ and $[k:k^{p^n}]$ is finite for all $n \geq 1$. The
following characterization of semi-perfect fields was pointed out to
us by Michel Brion.

\begin{lemma} \label{lm:SemiPerfect}
A field $k$ of characteristic $p>0$ satisfies $[k:k^{p^n}]=[k:k^p]^n$ for
every $n \geq 0$ and is therefore semi-perfect if and only if $[k:k^p]$
is finite.
\end{lemma}

(The index $[k:k^p]$ is called the {\em imperfect degree} of $k$ in
\cite[Section 2.7]{FriedJarden}.)

\begin{proof}
For $n>0$ we have
\[ [k:k^{p^n}]=[k:k^p]\cdot[k^p:k^{p^n}]. \]
The Frobenius map $k \to k^p,\ x \mapsto x^p$ is a field
isomorphism that maps $k^{p^{n-1}}$ onto $k^{p^n}$, so the right-hand
side equals $[k:k^p] \cdot [k:k^{p^{n-1}}]$, and we are done by
induction.
\end{proof}

The field $\bF_p(t_1, \ldots, t_m)$ is semi-perfect for any $m \ge 0$, but is only perfect when $m=0$. The field $\bF_p(t_1, t_2, \ldots)$ is not semi-perfect. The semi-perfect condition appeared in \cite{ESS1, ESS2}, and will show up in \S \ref{ss:astr} below for similar reasons. Due to this, the following observation is important:

\begin{proposition} \label{prop:semiperfect}
A Brauer field is semi-perfect.
\end{proposition}

\begin{proof}
It suffices to treat the case of positive characteristic, so suppose
$k$ is a Brauer field of characteristic $p>0$. Set $m=N_k(p)$. We
claim that $[k:k^p] \leq m$. Indeed, suppose that $a_1, \ldots,
a_{m+1}$ are elements of $k$. By the definition of $N_k$ the equation
$a_1 x_1^p+\ldots+a_{m+1} x_{m+1}^p=0$ has a nonzero solution.
This is exactly a non-trivial linear dependence of $a_1, \ldots,
a_{m+1}$ over the field $k^p$, which proves the claim. The proposition
now follows from Lemma~\ref{lm:SemiPerfect}.
\end{proof}

\section{Background on strength} \label{s:prelim}

\subsection{Basic definitions}

A \defn{multi-degree} is a tuple $\ul{d}=(d_1, \ldots, d_r)$ of integers
with $d_1 \ge d_2 \ge \cdots \ge d_r \geq 1$. The \defn{length}
of a multi-degree, denoted $\len(\ul{d})$, is the length of the
tuple. We compare multi-degrees lexicographically. The collection of
all multi-degrees is well-ordered.

For a $k$-vector space $V$, we let $\cP(V)$ be the space polynomial functions
on $V$. Since we assume that $k$ is infinite, $\cP(V)$ can be identified with 
the symmetric algebra $\Sym(V^*)$. We let $\cP_d(V)=\Sym^d(V^*)$ be the space of degree $d$ polynomial functions on $V$, and put $\cP_{\ul{d}}(V) = \cP_{d_1}(V) \times \cdots \times \cP_{d_r}(V)$.

For an element $\ul{f} \in \cP_{\ul{d}}(V)$, we write $Z(\ul{f})$ for
the closed subvariety of $V$ defined by the equations $f_i=0$ for $1 \le i \le r$. Here we regard $V$ not just
as a vector space over $k$ but also as an affine scheme over $k$. This
will not lead to confusion.

\subsection{Strength} \label{ss:strength}

We recall the definition of strength. It seems that Schmidt \cite{Schmidt} was the first to study this concept. It has since been rediscovered several times. We learned the ideas through the work of Ananyan and Hochster \cite{AH}, and therefore use their terminology.

\begin{definition}
The \defn{strength} of a homogeneous polynomial $f \in \cP_d(V)$ of degree $d>0$, denoted $\str(f)$, is the minimal nonnegative
integer $s$ for which there is an expression
\[
f = g_1h_1+ \cdots g_s h_s,
\]
where $g_i$ and $h_i$ are homogeneous polynomials on $V$ of degrees
$<d$; if no such expression exists (which happens only when $d=1$),
then $\str(f)$ is defined to be $\infty$. The \defn{strength} of a
tuple $\ul{f} \in \cP_{\ul{d}}(V)$, denoted $\str(\ul{f})$, is the
minimal strength of a nontrivial $k$-linear combination of polynomials
in $\ul{f}$ of the same degree. For the empty tuple, this minimum is
defined to be $\infty$.
\end{definition}

We now give a simple but useful fact about the behavior of strength under restriction.

\begin{proposition} \label{prop:restriction}
Let $\ul{f} \in \cP_{\ul{d}}(V)$, let $W$ be a linear subspace of $V$,
and let $\ul{f}|_W  \in \cP_{\ul{d}}(W)$ denote the tuple $(f_1|_W,\ldots,f_r|_W)$. Then
$\str(\ul{f}) \leq \str(\ul{f}|_W) + (\dim V - \dim W)$.
\end{proposition}

\begin{proof}
It suffices to prove this in the case where $\ul{d}=(d)$; write
$f=f_1$. Choose (linear) coordinates $x_1,\ldots,x_n$ on $V$ where $W$ is the
common zero set of $x_1,\ldots,x_c$ and hence
$x_{c+1}|_W,\ldots,x_n|_W$ are a set of coordinates on $W$. We
may write 
\begin{displaymath}
f=x_1 h_1(x_1,\ldots,x_n) + \ldots + x_c h_c(x_1,\ldots,x_n) +
\tilde{f}(x_{c+1},\ldots,x_n).
\end{displaymath}
Then $f|_W=\tilde{f}(x_{c+1}|_W,\ldots,x_n|_W)$ has the same
strength as $\tilde{f}$
and the equation above shows that $\str(f) \leq c + \str(\tilde{f})$. 
\end{proof}

Diagonal forms play a central role in our paper; the following gives a
lower bound for their strength. 

\begin{proposition} \label{prop:diagstrength}
A diagonal form $a_1 x_1^d + \ldots + a_n x_n^d$ of degree $d>0$ with
$a_1,\ldots,a_n \in k$ nonzero has strength at least $n/2$,
provided that $d$ is not a multiple of the characteristic of $k$.
\end{proposition}

\begin{proof}
Let $f=a_1 x_1^d + \ldots + a_n x_n^d$, and suppose that 
\[
f=g_1\cdot h_1 + \ldots + g_s\cdot h_s
\]
where the $g_i$ and $h_i$ are homogeneous of degree $<d$.
Differentiating with respect to $x_j$, we find 
\[ a_j d x_j^{d-1} = \sum_{i-1}^s \left(\frac{\partial g_i}{\partial x_j}
\cdot h_i + g_i \cdot \frac{\partial h_i}{\partial x_j} \right). \]
Since, by assumption, $a_jd$ is nonzero in $k$, we
find that the ideal generated by the homogeneous polynomials
$g_1,\ldots,g_s,h_1,\ldots,h_s$ contains all $x_j^{d-1}$. But an ideal
generated by $2s$ homogeneous polynomials defines a variety in $\bP_k^n$
of codimension $\leq 2s$, so $2s \geq n$.
\end{proof}

\subsection{Strength over field extensions} \label{ss:astr}

In our proofs we will often face the following situation: we have a
tuple $\ul{f}$ of homogeneous polynomials over $k$ and have assumed
that $\str(\ul{f})$ is large, but need the more geometric fact that the
strength of $\ul{f}$ is large when $\ul{f}$ is regarded as a tuple
of polynomials over an algebraic closure $\ol{k}$ of $k$.
We therefore define the \defn{absolute strength} of $\ul{f}$, denoted
$\astr(\ul{f})$, to be the strength of $\ul{f}$ over $\ol{k}$. 
The following example shows that $\str(\ul{f})$ and $\astr(\ul{f})$
can be arbitrarily far apart.

\begin{example}
Let $k=\bF_p(t_1,t_2,\ldots)$, where the $t_i$ are algebraically
independent over $\bF_p$. Then the polynomial
\[
f= t_1 x_1^p +\ldots + t_n x_n^p  =\left( t_1^{1/p} x_1 + \ldots + t_n^{1/p} x_n\right)^p
\]
has $\astr(f)=1$. But
$f$ has $\str(f) \geq n/2$ over $k$. Indeed, a variant of the argument
used in Proposition~\ref{prop:diagstrength} applies: suppose that 
\[
f=g_1 \cdot h_1 + \ldots + g_s\cdot h_s, \]
where $g_i,h_i \in k[x_1,\ldots,x_n]$ are homogeneous polynomials
of positive degrees. Only finitely many of the $t_j$ appear in the
coefficients of $g_i,h_i$, say $t_1,\ldots,t_m$.  We write
$g_i=g_i(t,x),h_i=h_i(t,x)$
where $t=(t_1,\ldots,t_m)$ and $x=(x_1,\ldots,x_n)$.  There exists a
finite extension $\bF_q$ of $\bF_p$ and a point $c \in \bF_q^m$ such that
the coefficients of the $g_i$ and the $h_i$, which are rational
functions in $t$, are all defined for $t=c$. For
$j=1,\ldots,n$ we have 
\[ x_j^p = \frac{\partial f}{\partial t_j} (c,x) = \sum_{i=1}^s
\left(\frac{\partial g_i}{\partial t_j} (c,x) h_i(c,x) +
g_i(c,x) \frac{\partial h_j}{\partial t_j} (c,x) \right)  \in \bF_q [x_1,\ldots,x_n] \]
and we conclude that $x_j^p$ lies in the ideal of $\bF_q[x_1,\ldots,x_n]$ generated by the
homogeneous polynomials $g_i(c,x),h_i(c,x)$ for $i=1,\ldots,s$. As in
the proof of Proposition~\ref{prop:diagstrength}, we conclude that $2s
\geq n$. We actually expect that $f$ has strength $n$ over $k$, but we do not have a proof.
\end{example}

The following theorem implies that examples like $f$
above only exist over non-semi-perfect fields.

\newC{c:astr}
\begin{theorem} \label{thm:astr}
Fix a multi-degree $\ul{d}$, a semi-perfect field $k$, and a nonnegative
integer $s$. Then for any $k$-vector space $V$ and any $\ul{f} \in
\cP_{\ul{d}}(V)$ with $\astr(\ul{f}) \leq s$ we have $\str(\ul{f})
\leq \C{c:astr}(\ul{d},s)$.
\end{theorem}

When $k$ has characteristic $0$ or characteristic strictly larger than
the entries of $\ul{d}$, this theorem is proved in \cite{LZ2} (over
certain fields) and the forthcoming paper \cite{BDLZ}. The general case follows from forthcoming work on the geometry of polynomial representations in
positive characteristic \cite{BDS2}. 

\subsection{The work of Ananyan--Hochster}

In the course of proving Stillman's conjecture \cite{AH}, Ananyan and Hochster proved that polynomials of high strength have numerous nice properties. See \cite{ESS3} for a general exposition on this theme. We will require two specific such results, which we now review.

Recall that elements $f_1, \ldots, f_r$ in a ring form a \defn{prime
sequence} if the ideal $I_i=(f_1, \ldots, f_i)$ is prime for every $i
\in \{1,\ldots,r\}$ and satisfies $I_i \subsetneq I_{i+1}$ for every
$i \in \{1,\ldots,r-1\}$. This implies that $f_1, \ldots, f_r$ is a
regular sequence.

\newC{c:AH}
\begin{theorem} \label{thm:AH}
Suppose $k$ is semi-perfect. If $\str(\ul{f})>\C{c:AH}(\ul{d})$, then $\ul{f}$ is a prime sequence.
\end{theorem}

\begin{proof}
This is \cite[Theorem~A]{AH} if $k$ is algebraically closed;
let the corresponding bound be $\widetilde{\C{c:AH}}(\ul{d})$. For the general
case, set
$\C{c:AH}(\ul{d})=\C{c:astr}(\ul{d},\widetilde{\C{c:AH}}(\ul{d}))$,
where $\C{c:astr}(\ul{d},s)$ is defined in
Theorem~\ref{thm:astr}. Then if $f$ has $\str(f) > \C{c:AH}(\ul{d})$, we have
$\astr(f) > \widetilde{\C{c:AH}}(\ul{d})$, so that $f_1,\ldots,f_r$ is a prime
sequence in $\ol{k}[x_1,\ldots,x_n]$. It is then also a prime sequence
in $k[x_1,\ldots,x_n]$.
\end{proof}

The next result compares the strength of $f$ to the codimension of the
singular locus of $Z(f)$. See \cite{KLP} for a quantitative result.

\begin{theorem} \label{thm:sing}
Suppose $k$ is semi-perfect. Let $f \in \cP_d(V)$ and let $I \subseteq
\cP(V)$ be the ideal generated by the first-order partial derivatives of $f$. Then
$\str(f)$ can be lower-bounded in terms of $d$ and the codimension of
$Z(I)$; and conversely, that codimension can be lower-bounded in terms of $d$
and $\str(f)$.
\end{theorem}

\begin{proof}
If $k$ is algebraically closed, this is \cite[Theorem~F]{AH}. The semi-perfect case follows from this and Theorem~\ref{thm:astr}.
\end{proof}

\subsection{Regularization}

Let $\ul{f} \in \cP_{\ul{d}}(V)$ be given. Consider the following (vague) algorithm:
\begin{itemize}[leftmargin=4.5em]
\item[Step 1.] If $\ul{f}$ has high strength, then halt.
\item[Step 2.] Otherwise, some nontrivial linear combination
$u=c_a f_a + \ldots + c_b f_b$ of forms in $\ul{f}$ of the same degree
has small strength, and we may assume $c_b \neq 0$. Then write 
$u=g_1h_1+\cdots+g_sh_s$ with $s$ small, replace $f_b$ with  
$g_1, h_1, \ldots, g_s, h_s$, reorder to get the degrees in
non-increasing order, and return to Step~1.
\end{itemize}
The end result of this procedure is a high strength tuple $\ul{g}=(g_1, \ldots, g_s)$ such that each $f_i$ can be expressed as polynomial in $g_1, \ldots, g_s$.

This procedure is called \defn{regularization} in \cite{LZ1}. It is a key element in many arguments involving strength: one first proves results in the high strength regime, and then uses regularization to say something about arbitrary polynomials. For example, this is the basic structure of the proof of Stillman's conjecture in \cite{AH}.

The regularization procedure is important in this paper. The following is a precise statement of the version we use. What is particularly important here is that we obtain control on the length of the output tuple $\ul{g}$.

\begin{proposition} \label{prop:small-subalg}
Let $\ul{d}$ be a fixed multi-degree, and let $\Phi$ be any function
from the set of multi-degrees to $\bN$. Then there exists $s \in \bN$
with the following property. If $\ul{f} \in \cP_{\ul{d}}(V)$, then
there exist a multi-degree $\ul{e} \leq \ul{d}$ with $\len(\ul{e})
\leq s$ and a tuple $\ul{g} \in \cP_{\ul{e}}(V)$ such that
$\str(\ul{g}) > \Phi(\ul{e})$, and each $f_i$ belongs
to the algebra $k[\ul{g}]$ generated by the $g_j$'s.
\end{proposition}

\begin{proof}
The same argument used in \cite[Proposition~8.1]{ESS2} applies. Also see \cite[\S2] {Schmidt} and \cite[Theorem~B]{AH}.
\end{proof}

\begin{remark}
The regularization procedure is wildly inefficient: the length of the output tuple $\ul{g}$ will be worse than tower exponential in the length of the input tuple $\ul{f}$. See \cite{Wooley3} for details. This implies that many of the constants we find in this paper are similarly bad. We note that \cite{LZ1} gives an improvement of the regularization procedure in a certain context that has only polynomial growth.
\end{remark}

\section{Reduction to Proposition~\ref{prop:good}} \label{s:red}

Throughout this section $k$ denotes a fixed infinite Brauer field, and $\ul{d}=(d_1, \ldots, d_r)$ denotes a multi-degree of length $r$.

\subsection{Set-up}

Consider the following statement:
\begin{description}[align=right,labelwidth=1.4cm,leftmargin=!]
\item[$\Sigma(\ul{d})$] There exists a quantity
$\C{c:mainthm}(\ul{d})$ such that if $\ul{f} \in \cP_{\ul{d}}(V)$
satisfies $\str(\ul{f}) > \C{c:mainthm}(\ul{d})$, then the
$k$-points of $Z(\ul{f})$ are Zariski dense.
\end{description}
In the above statement, $V$ is allowed to be any finite dimensional $k$-vector space. Theorem~\ref{mainthm} states that $\Sigma(\ul{d})$ holds for all multi-degrees $\ul{d}$. We will prove $\Sigma(\ul{d})$ by induction on $\ul{d}$. To this end, it will be convenient to introduce the following auxiliary statement:
\begin{description}[align=right,labelwidth=1.4cm,leftmargin=!]
\item[$\Sigma^*(\ul{d})$] the statement $\Sigma(\ul{e})$ holds for all multi-degrees $\ul{e}<\ul{d}$.
\end{description}
To prove the theorem, it suffices to show that $\Sigma^*(\ul{d})$ implies $\Sigma(\ul{d})$.

\subsection{Existence of good subspaces}

Let $f \in \cP_d(V)$. We say that subspaces $E_1, \ldots, E_n$ of $V$ are \defn{$f$-orthogonal} if
\begin{displaymath}
f(v_1+\ldots+v_n)=f(v_1)+\ldots+f(v_n)
\end{displaymath}
holds for all $v_i \in E_i$. If $\dim{V}=3$, then we say that $f$ is \defn{good} if it has the form
\begin{displaymath}
xy^{d-1}+ay^d+bz^d,
\end{displaymath}
where $a\in k$ and $b \in k^\times$, and $x$, $y$, and $z$ are coordinates
on $V$.

The following is the key intermediate result we need to prove our
main theorems. 

\newC{c:good}
\begin{proposition} \label{prop:good}
Assume $\Sigma^*(\ul{d})$. Let $\ul{f} \in \cP_{\ul{d}}(V)$, let $1
\le i \le r$, and let $F$ be a subspace of $V$. If $\str(\ul{f})>\C{c:good}(\ul{d}, \dim{F})$, then there is a three dimensional subspace $E$ of $V$ such that:
\begin{enumerate}
\item $E+F=E\oplus F$,
\item $E$ and $F$ are $f_j$-orthogonal for $i \le j \le r$,
\item $f_j$ vanishes on $E$ for $i<j \le r$, and 
\item the restriction of $f_i$ to $E$ is good.
\end{enumerate}
\end{proposition}

The proof of Proposition~\ref{prop:good} is deferred to \S \ref{s:normal0}
and \S \ref{s:normalp}; in the remainder of this section we use it to prove
our main theorems. 

\subsection{Non-vanishing on small subspaces}

Suppose $g$ is a polynomial that does not vanish identically on
$Z(\ul{f}) \subset V$; we will abbreviate this requirement to ``$g$
is non-vanishing on $Z(\ul{f})$''. We will need to know that one can
pass to a small-dimensional linear subspace of $V$ and maintain the
non-vanishing of $g$ on the zero locus of $\ul{f}$. This is established
in the following proposition.

\begin{proposition} \label{prop:non-vanishing}
Let $Z$ be a closed subvariety of $V$ of codimension $\le r$, and let $g \in \cP(V)$ be a polynomial that does not vanish identically on $Z$. Then there is a vector subspace $F$ of $V$ of dimension $\le r+1$ such that $g$ does not vanish identically on $F \cap Z$.
\end{proposition}

\begin{proof}
Replacing $Z$ with an irreducible component on which $g$ is non-vanishing, we assume that $Z$ is irreducible. By Noether normalization, we can find a subspace $W$ of $V$ of dimension $\le r$ such that the projection map $\pi \colon V \to V/W$ is finite when restricted to $Z$. Here we are using the fact that $k$ is infinite to achieve normalization with a linear map.

Let $Z' \subset Z$ be the zero locus of $g$ on $Z$. Since $Z$ is irreducible, it follows that $Z'$ has codimension one in $Z$. Since $\pi$ is finite, it follows that $\pi(Z')$ is a closed subset of $V/W$ of codimension~1; let $U$ be the complement of $\pi(Z')$, which is a dense open subset. Let $x$ be a $k$-point in $U$; such a point exists since $k$ is infinite. Then $\pi^{-1}(x)$ contains a $\ol{k}$-point $y$ of $Z$ at which $g$ does not vanish. We can therefore take $F=\pi^{-1}(\operatorname{span}(x))$, as this is subspace of $V$ of dimension $\le r+1$ containing $y$ over $\ol{k}$. 
\end{proof}

\subsection{A normal form}

We now show that a tuple of polynomials can be put into a rather simple form on an appropriate linear subspace.

\newC{c:normal}
\begin{proposition} \label{prop:normal}
Suppose $\Sigma^*(\ul{d})$ holds. Let $\ul{f} \in \cP_{\ul{d}}(V)$ and
let $g \in \cP(V)$ be non-vanishing on $Z(\ul{f})$. If $\str(\ul{f}) >
\C{c:normal}(\ul{d})$, then there exists a subspace $W$ of $V$ with
coordinates $\{x_i,y_i,z_i\}_{1 \le i \le r} \cup \{w_j\}_{1 \le j \le
m}$ such that
\begin{enumerate}
\item for every $i\in\{1,\ldots,r\}$ the restriction of $f_i$ to $W$ has the form
\begin{displaymath}
x_iy_i^{d_i-1}+a_i y_i^{d_i}+b_i z_i^{d_i}+h_i(x_{i+1},y_{i+1},z_{i+1}, \ldots, x_r, y_r, z_r, w_1, \ldots, w_m)
\end{displaymath}
for some polynomial $h_i$ and scalars $a_i \in k, b_i\in k^\times$; and
\item $g$ does not vanish identically on $Z(\ul{f}) \cap W$.
\end{enumerate}
\end{proposition}

\begin{proof}
Let $F$ be a subspace of $V$ of dimension $\le r+1$ such that $Z(\ul{f}) \cap F$ is non-empty, and $g$ does not vanish identically on $Z(\ul{f}) \cap F$. This exists by Proposition~\ref{prop:non-vanishing}. We construct three-dimensional subspaces $E_1, \ldots, E_r$ of $V$ with the following properties:
\begin{enumerate}[(i)]
\item $E_1 + \ldots + E_r + F = E_1 \oplus\cdots\oplus E_r\oplus F$.
\item The spaces $E_i, \ldots, E_r, F$ are $f_j$-orthogonal for $i \le j \le r$.
\item The forms $f_{i+1}, \ldots, f_r$ vanish identically on $E_i$.
\item The restriction of $f_i$ to $E_i$ is good, say $x_i y_i^{d_i-1}
+ a_i y_i^{d_i} + b_i z_i^{d_i}$, where $x_i, y_i, z_i$ are suitable
coordinates on $E_i$, and $a_i \in k, b_i\in k^\times$.
\end{enumerate}
Granted this, we can take $W=E_1 \oplus \cdots \oplus E_r \oplus F$ and take $w_1, \ldots, w_m$ to be coordinates on~$F$.

To construct the $E_i$'s we iteratively apply
Proposition~\ref{prop:good}. More precisely, we first construct $E_r$ ($f_r$-orthogonal to $F$), then we construct $E_{r-1}$ ($f_j$-orthogonal to $E_r \oplus F$ for $r-1\leq j\leq r$), and so on, until we finally construct $E_1$ ($f_j$-orthogonal to $E_2 \oplus \cdots \oplus E_r \oplus F$ for $1\leq j\leq r$). 

To construct $E_r$, we must have $\str(\ul{f})>\C{c:good}(\ul{d}, r+1)$, since $\dim(F) \le r+1$. To construct $E_{r-1}$, we must have $\str(\ul{f})>\C{c:good}(\ul{d}, r+4)$, since $\dim(E_r \oplus F) \le r+4$. The pattern continues in this way. We thus see that to carry out the entirely construction, we must have $\str(\ul{f})>\C{c:good}(\ul{d}, 4r-2)$. We can therefore take $\C{c:normal}(\ul{d})=\C{c:good}(\ul{d}, 4r-2)$. (Here we have assumed that $\C{c:good}(\ul{d},n)$ is increasing in $n$, which can always be arranged.)
\end{proof}

We observe some consequences of the normal form.

\begin{proposition} \label{prop:normal2}
Suppose $\Sigma^*(\ul{d})$ holds. Let $\ul{f} \in \cP_{\ul{d}}(V)$ and $g \in \cP(V)$ and $W$ be as in the previous proposition.
\begin{enumerate}
\item The restriction of $\ul{f}$ to $W$ is a prime sequence. In particular, $Z(\ul{f}) \cap W$ is irreducible.
\item The open subset $U$ of $Z(\ul{f}) \cap W$ defined by $y_i \ne 0$ for all $1 \le i \le r$ is non-empty, and isomorphic to an open subvariety of affine space of dimension $2r+m$.
\item The $k$-points of $Z(\ul{f}) \cap W$ are dense.
\end{enumerate}
\end{proposition}

\begin{proof}
For (a), let $I_i$ be the ideal generated by $f_1|_W,\ldots,f_i|_W$.
Set 
\[ \omega=(1,1,1,2,2,2,\ldots,r,r,r,(r+1),\ldots,(r+1)) \in \bZ^{3r+m}. \] 
Then the initial ideal of $I_i$ with respect to the weight vector
$\omega$ is the ideal generated by $x_jy_j^{d_j-1}+a_j y_j^{d_j} + b_j
z_j^{d_j}$ with $j=1,\ldots,i$. Since the latter ideal is 
prime (here we use that $b_j \neq 0$), so is the former.

For (b), we note that the projection forgetting $x_1,\ldots,x_r$
induces an isomorphism from $U$ to the open subspace $\bG^r_{\bm}
\times \bA^r \times \bA^m$ of the affine space $\bA^{r + r + m}$ with coordinates
$y_1,\ldots,y_r,z_1,\ldots,z_r,w_1,\ldots,w_r$.

By (b), the $k$-points of $U$ are dense in $U$ (since $k$ is infinite), and by (a), the open
set $U$ is dense in $Z(\ul{f}) \cap W$. Therefore the $k$-points of $U$ are dense in $Z(\ul{f}) \cap W$. This shows (c).
\end{proof}

\subsection{Proof of Theorem~\ref{mainthm}}

Let $\ul{d}$ be a multi-degree and suppose $\Sigma^*(\ul{d})$ holds. We show that $\Sigma(\ul{d})$ holds with $\C{c:mainthm}(\ul{d})=\C{c:normal}(\ul{d})$, where $\C{c:normal}$ is as in Proposition~\ref{prop:normal}.

To this end, let $\ul{f} \in \cP_{\ul{d}}$ be given with $\str(\ul{f})
> \C{c:normal}(\ul{d})$. We must show that the $k$-points of $Z=Z(\ul{f})$ are Zariski dense. Let $g$ be a polynomial on $V$ that is not identically zero on~$Z$. We construct a $k$-point of $Z$ at which $g$ is nonzero; since $g$ is arbitrary, this will prove that $Z(k)$ is dense. Let $W$ be as in Proposition~\ref{prop:normal}. Since $g$ does not vanish identically on $Z \cap W$, and the $k$-points of $Z \cap W$ are dense by Proposition~\ref{prop:normal2}(c), it follows that there is a $k$-point of $Z \cap W$ at which $g$ is nonzero. This completes the proof.

\subsection{Proof of Theorem~\ref{mainthm2}} \label{ss:pf-mainthm2}

For a multi-degree $\ul{e}$, set $\Phi(\ul{e})=\C{c:mainthm}(\ul{e})$. Fix a multi-degree~$\ul{d}$, and let $s=s(\ul{d})$ be as in Proposition~\ref{prop:small-subalg}. Let $\ul{f} \in \cP_{\ul{d}}(V)$ be given, and put $Z=Z(\ul{f})$. By Proposition~\ref{prop:small-subalg}, there exists $\ul{g} \in \cP_{\ul{e}}(V)$ with $\len(\ul{e}) \le s$ such that $\str(\ul{g}) \ge \C{c:mainthm}(\ul{e})$ and each $f_i$ belongs to $k[\ul{g}]$. Let $Z'=Z(\ul{g})$, and note that $Z' \subset Z$. Since $Z'$ is defined by $\leq s$ equations, $Z'$ has codimension $\leq s$ in $V$. By Theorem~\ref{mainthm}, $Z'(k)$ is dense in $Z'$. We thus see that the Zariski closure of $Z(k)$ contains $Z'$, and therefore has codimension $\leq s$ in $V$, and thus in $Z$ as well. We can therefore take $\C{c:mainthm2}(\ul{d})=s$.

\subsection{A variant}

Theorem~\ref{mainthm} states that the $k$-points of $Z(\ul{f})$ are dense if $\ul{f}$ has high strength. We now prove a variant of this result: with only a partial strength hypothesis, we show that certain functions are non-vanishing at $k$-points. This proposition plays an important role in \S \ref{s:normal0} and \S \ref{s:normalp}. The proof is similar to that of Theorem~\ref{mainthm2}, but makes crucial use of results from \cite{AH}.

\newC{c:lowdeg} \newC{c:codim}
\begin{proposition} \label{prop:lowdeg}
Suppose $\Sigma(\ul{d})$ and $\Sigma^*(\ul{d})$ hold. Let $i,d'$ be such that $d_1, \ldots, d_i \ge d'$ and $d_{i+1}, \ldots, d_r<d'$. Let $\ul{f} \in \cP_{\ul{d}}(V)$ and $g \in \cP_{d'}(V)$, and suppose $\str(f_1, \ldots, f_i, g)>\C{c:lowdeg}(\ul{d},d')$. Then there exists a closed subvariety $Z'$ of $Z=Z(\ul{f})$ of codimension $ \le \C{c:codim}(\ul{d}, d')$ such that $Z'(k)$ is dense in $Z'$ and $g$ is not identically zero on $Z'$. In particular, there is a $k$-point of $Z$ at which $g$ does not vanish.
\end{proposition}

\begin{proof}
We first give the basic idea. Express $f_{i+1}, \ldots, f_r$ in terms of high strength polynomials $(f'_1, \ldots, f'_t)$ using Proposition~\ref{prop:small-subalg}. Then $g, f_1, \ldots, f_i, f'_1, \ldots, f'_t$ has high strength, and so it is a prime sequence (by Theorem~\ref{thm:AH}) and the $k$-points of $Z'=Z(f_1, \ldots, f_i, f'_1, \ldots, f'_t)$ are dense (by $\Sigma$). The prime sequence statement implies that $g$ is non-vanishing on $Z'$.

We now give a rigorous proof, which simply amounts to combining the above idea with some careful bookkeeping. Put $\ul{d}^+=(d_1, \ldots, d_i)$ and $\ul{d}^-=(d_{i+1}, \ldots, d_r)$. For a multi-degree~$\ul{e}$, put
\begin{displaymath}
\Phi(\ul{e}) = \max\{\C{c:mainthm}(\ul{d}^+ \cup \ul{e}), \C{c:AH}(\ul{d}^+ \cup (d') \cup \ul{e})\}
\end{displaymath}
Here $\C{c:mainthm}$ and $\C{c:AH}$ and are as in
Theorems~\ref{mainthm} and~\ref{thm:AH}, and we use $\cup$ to denote
the operation of concatenating tuples and sorting the result. 
Let $s$ be the number produced by Proposition~\ref{prop:small-subalg} for the multi-degree $\ul{d}^-$ and this $\Phi$. Define
\begin{displaymath}
\C{c:lowdeg}(\ul{d},d') = {\textstyle\max_{\ul{e}}} \Phi(\ul{e}),
\end{displaymath}
where the maximum is taken over the finitely many $\ul{e}$'s with $\ul{e} \le \ul{d}^-$ and $\len(\ul{e}) \le s$. We show that the proposition holds with this value of $\C{c:lowdeg}(\ul{d},d')$.

Thus let $\ul{f} \in \cP_{\ul{d}}(V)$ be given with $\str(f_1, \ldots, f_i, g) 
> \C{c:lowdeg}(\ul{d},d')$. Applying
Proposition~\ref{prop:small-subalg}, there exists $\ul{f}' \in
\cP_{\ul{e}}(V)$ with $\ul{e} \le \ul{d}^-$ and $t=\len(\ul{e}) \le
s$ such that $\str(\ul{f}') > \Phi(\ul{e})$ and $f_{i+1}, \ldots, f_r$ belongs to $k[\ul{f}']$. We have
\begin{displaymath}
\str(f_1, \ldots, f_i, f'_1, \ldots, f'_t)=\min\{\str(f_1, \ldots,
f_i),\str(f'_1, \ldots, f'_t)\} > \Phi(\ul{e}) \ge \C{c:mainthm}(\ul{d}^+ \cup \ul{e})
\end{displaymath}
as the polynomials in $(f_1, \ldots, f_i),(f'_1, \ldots, f'_t)$ are of disjoint degrees and the strength of $(f_1, \ldots, f_i)$ is at least the strength of $(f_1, \ldots, f_i, g)$. Note that $\ul{d}^+ \cup \ul{e} \le \ul{d}$, and so $\Sigma(\ul{d}^+ \cup \ul{e})$ holds. Hence the $k$-points of $Z'=Z(f_1, \ldots, f_i, f'_1, \ldots, f'_t)$ are Zariski dense.

Similarly, we have
\begin{displaymath}
\str(f_1, \ldots, f_i, g, f'_1, \ldots, f'_t) > \Phi(\ul{e}) \ge
\C{c:AH}(\ul{d}^+ \cup (d') \cup \ul{e}),
\end{displaymath}
and so $(f_1, \ldots, f_i, g, f'_1, \dots, f'_t)$ is a prime sequence by
Theorem~\ref{thm:AH}. But any permutation of a prime sequence is a prime
sequence, hence $(f_1,\ldots,f_i,f'_1,\ldots,f'_t,g)$ is also a prime
sequence, and it follows that $g$ does not vanish identically on $Z'$. Note that the codimension of $Z'$ in $V$ is at most $i+t$, and so we can take $\C{c:codim}=r+s$.
\end{proof}

\section{Proof of Proposition~\ref{prop:good}: characteristic~0} \label{s:normal0}

Throughout \S \ref{s:normal0}, we fix an infinite Brauer field $k$ and let $\ul{d}=(d_1, \ldots, d_r)$ be a multi-degree.

\subsection{Overview}

The purpose of this section is to prove Proposition~\ref{prop:good} when $k$ has characteristic~0. To do this, we prove the following two independent statements:
\begin{enumerate}
\item A diagonal form of sufficiently high rank admits a good
specialization. 
\item Given $\ul{f} \in \cP_{\ul{d}}(V)$ of sufficiently high strength
and a subspace $F$ of $V$ of low dimension, we can find a line $L$ in
$V$ having a number of desirable properties, including
$f_j$-orthogonality to $F$ for a range of $j$'s.
\end{enumerate}
To prove Proposition~\ref{prop:good}, we apply (b) repeatedly to construct lines $L_1, \ldots, L_n$. The restriction of $f_j$ to $L_1 \oplus \cdots \oplus L_n$ is diagonal, and by (a) it admits a good specialization.

\subsection{Good specialization of diagonal forms}

We say that a polynomial $f$ on a vector space $V$ is \defn{diagonal}
of rank $r$ if there are coordinates $x_1, \ldots, x_n$ on $V$ such
that 
\[
f=a_1 x_1^d+\ldots+a_r x_r^d
\]
with $a_1, \ldots, a_r\in k$ all
nonzero. 

\newC{c:diag}
\begin{proposition} \label{prop:diag}
Suppose $\Sigma^*(d)$ holds and $k$ has characteristic~0. Let $f \in
\cP_d(V)$ be a diagonal form of rank $> \C{c:diag}(d)$. Then there is
a three dimensional subspace $E$ of $V$ such that $f \vert_E$ is good.
\end{proposition}

\begin{proof}
Put
\begin{displaymath}
m = \max\{\C{c:mainthm}((d-2,d-3, \ldots, 1)), \C{c:AH}((d,d-1,
\ldots, 1))\} + 1,
\end{displaymath}
where $\C{c:mainthm}$ is the constant in the statement $\Sigma(\ul{d})$ and $\C{c:AH}$ is the constant in Theorem~\ref{thm:AH}. We show that the proposition holds with 
\[
\C{c:diag}(d) = (N_k(d)+1) \cdot m.
\]
Let $\{e_i\}_{1 \le i \le n}$ be a basis for $V$ in which $f$ is
diagonal, say $f=\sum_{i=1}^{r+1} c_i z_i^d$ where all the $c_i$ are
nonzero and $r \ge \C{c:diag}(d) $. We set aside the $(r+1)$st term in
the sum; in what follows, we work in the space
$V_0=\operatorname{span}(e_1, \ldots, e_r)$. We have $f
\vert_{V_0}=\sum_{i=1}^r c_i z_i^d$. Break up the sum into at least
$m$ blocks, each of which has size $>N_k(d)$. In each block, we
can find a nonzero vector on which $f$ vanishes (by definition of
$N_k(d)$), and this vector has at least two nonzero entries. Putting
these vectors together, we obtain a vector $v \in V_0$ such that
$f(v)=0$ and $v$ has at least $2m$ nonzero entries. 

Let $w=(w_1,\ldots,w_n)$ be a second vector in $V_0$ (to be determined). We have
\begin{displaymath}
f(xv+yw) = \sum_{j=1}^d f_j(w) x^{d-j} y^j, \qquad
f_j(w) = \binom{d}{j} \cdot \sum_{k=1}^r c_k v_k^{d-j} w_k^j
\end{displaymath}
As a function of $w$, each polynomial $f_j$ is a diagonal form of rank
$\ge 2m$, and thus strength $\ge m$ by
Proposition~\ref{prop:diagstrength}. Since the $f_j$'s have different
degrees, it follows that the tuple $(f_d, \ldots, f_1)$ has strength
$\ge m > \C{c:AH}((d,\ldots,1))$, and is therefore a prime
sequence by Theorem~\ref{thm:AH}. In particular, $f_d \cdot f_{d-1}$ is
non-vanishing on the irreducible variety $Z(f_1, \ldots, f_{d-2})$.
Since furthermore 
$m>\C{c:mainthm}((d-2,d-3,\ldots,1))$, we can find a $w \in V_0$ such
that $f_j(w)=0$ for $1 \le j \le d-2$ while $c=f_{d-1}(w)$ and $a=f_d(w)$ are non-zero. We
then have $f(xv+yw)=c xy^{d-1}+a y^d$. Note that we can make $c=1$ by
replacing $v$ with a scalar multiple. Also, the vectors $v$ and $w$
are linearly independent since $f(v)=0$ and $f(w)=a \ne 0$. Since
furthermore $f(e_{r+1})\neq 0$, we find that $f$ is good on $E=\operatorname{span}(v,w,e_{r+1})$. 
\end{proof}

\subsection{Construction of lines}

We now construct certain nice lines. The following proposition and
proof are actually also valid in positive characteristic, and will be
re-used in \S \ref{s:normalp}. 

\newC{c:orth-line}
\begin{proposition} \label{prop:orth-line}
Assume $\Sigma^*(\ul{d})$, let $\ul{f} \in \cP_{\ul{d}}(V)$, let $1 \le i \le r$, and let $F$ be a subspace of~$V$ of dimension $n$. If $\str(\ul{f})>\C{c:orth-line}(\ul{d}, n)$, then there exists a line $L \subset V$ such that
\begin{enumerate}
\item $L+F=L\oplus F$,
\item $L$ is $f_j$-orthogonal to $F$ for all $i \le j \le r$,
\item $f_j \vert_L=0$ for $i<j \le r$, and
\item $f_i \vert_L \ne 0$.
\end{enumerate}
\end{proposition}

\begin{proof}
Let $V'$ be a complementary space to $F$ and let $w_1, \ldots, w_n$ be a basis for $F$. Write
\begin{displaymath}
f_i(v+x_1 w_1+\ldots+x_n w_n) = {\textstyle\sum_{\ul{a}}} f_{i,\ul{a}}(v) x^{\ul{a}}
\end{displaymath}
where the sum is over multi-indices $\ul{a}=(a_1, \ldots, a_n)$ and $x^{\ul{a}}=x_1^{a_1} \cdots x_n^{a_n}$. We aim to find a $v \in V'$ satisfying the following conditions:
\begin{enumerate}[(i)]
\item $f_{j,\ul{a}}(v)=0$ for $i \le j \le r$ and $0<\vert \ul{a}
\vert<d_j$, where $\vert \ul{a} \vert=a_1+\ldots+a_n$; 
\item $f_j(v)=0$ for $i<j \le r$; and 
\item $f_i(v) \ne 0$.
\end{enumerate}
Given such a $v$, we can then take $L=\operatorname{span}(v)$. Indeed, since we have taken $v \in V'$, condition (a) holds. The orthogonality conditions in (b) are equivalent to the conditions in (i), while the remaining two conditions on $L$ match with (ii) and (iii).

The equations in (i) have degree $<d_j\leq d_i$. Thus, letting $\ul{e}$ be the
multi-degree of the equations in (i) and (ii), we see that no entry of
$\ul{e}$ exceeds $d_i$. Let $\ell$ be maximal such that $d_{\ell}=d_i$. The only equations in (i)
and (ii) of degree $d_i$ are those in (ii) with $i<j \le \ell$; all
others have degree $<d_i$. So there are strictly fewer entries in
$\ul{e}$ equal to $d_i$ than in $\ul{d}$. Thus $\ul{e}<\ul{d}$.  Let $s$ be the strength of the the tuple of
restrictions of $f_i, \ldots, f_{\ell}$ to $V'$. By
Proposition~\ref{prop:lowdeg} applied with $g=f_i|_{V'}$, we can find a $v \in V'$ satisfying (i), (ii), and (iii) provided that $s>\C{c:lowdeg}(\ul{e}, d_i)$.

Now, the strength of $f_i, \ldots, f_{\ell}$ on $V$ is at least
$\str(\ul{f})$. In restricting to $V'$, we lose at most $\dim{F}=n$ in
strength by Proposition~\ref{prop:restriction}. Thus $s \ge \str(\ul{f})-n$. Hence it suffices when $\C{c:orth-line}(\ul{d}, n)\geq \C{c:lowdeg}(\ul{e}, d_i)+n$. Note that $\ul{e}$ only depends on $\ul{d}$, $i$ and $n$. So we take $\C{c:orth-line}(\ul{d}, n)$ to be $\max_{i\in\{1,\ldots,r\}}\C{c:lowdeg}(\ul{e}, d_i)+n$.
\end{proof}

\subsection{Completion of proof}

We now prove Proposition~\ref{prop:good} in characteristic~0. We restate the result for convenience:

\begin{proposition} \label{prop:good0}
Assume $\Sigma^*(\ul{d})$ holds and $k$ has characteristic~0. Let $\ul{f} \in \cP_{\ul{d}}(V)$, let $1
\le i \le r$, and let $F$ be a subspace of $V$ of dimension $n$. If $\str(\ul{f})>\C{c:good}(\ul{d}, n)$, then there is a three dimensional subspace $E$ of $V$ such that:
\begin{enumerate}
\item $E+F=E\oplus F$,
\item $E$ and $F$ are $f_j$-orthogonal for $i \le j \le r$,
\item $f_j$ vanishes on $E$ for $i<j \le r$, and
\item the restriction of $f_i$ to $E$ is good.
\end{enumerate}
\end{proposition}

\begin{proof}
Let $m=\C{c:diag}(d_i)+1$, where $\C{c:diag}$ is given in Proposition~\ref{prop:diag}. We construct lines $L_1, \ldots, L_m$ in $V$ having the following properties:
\begin{enumerate}[(i)]
\item $L_1+ \ldots+ L_m+ F=L_1\oplus \cdots\oplus L_m\oplus F$,
\item $L_1, \ldots, L_m, F$ are $f_j$-orthogonal for $i \le j \le r$,
\item $f_j$ vanishes on $L_1,\ldots,L_m$ for $i<j \le r$, and 
\item $f_i$ is non-vanishing on each of $L_1,\ldots,L_m$.
\end{enumerate}
Assume for the moment we have this. The restriction of $f_i$ to $L_1 \oplus \cdots \oplus L_m$ is diagonal of rank $m$, and we can take $E \subset L_1 \oplus \cdots \oplus L_m$ to be the space produced by Proposition~\ref{prop:diag}.

To construct the $L_i$'s, we iteratively apply
Proposition~\ref{prop:orth-line}: we first construct $L_1$ orthogonal to $F$, then construct $L_2$ orthogonal to $L_1 \oplus F$, and so on. To construct $L_1$, we need $\str(\ul{f})>\C{c:orth-line}(\ul{d},n)$; to construct $L_2$, we need
$\str(\ul{f})>\C{c:orth-line}(\ul{d},n+1)$; continuing in this
manner, to construct $L_m$, we need
$\str(\ul{f})>\C{c:orth-line}(\ul{d}, n+m-1)$. We thus find
that we can take 
\[
\C{c:good}(\ul{d}, n)={\textstyle\max_{i\in\{1,\ldots,r\}}} \C{c:orth-line}(\ul{d},n+
\C{c:diag}(d_i)),
\]
where we assume, as we may, that $\C{c:orth-line}(\ul{d},n)$ is increasing in $n$. 
\end{proof}

\section{Proof of Proposition~\ref{prop:good}: characteristic \texorpdfstring{$p$}{p}} \label{s:normalp}

\subsection{Overview}

In this section we prove Proposition~\ref{prop:good} when $k$ has
characteristic~$p>0$. The strategy used in the previous section
does not quite work in this setting. Indeed, if $f$ is a diagonal
form of degree $p$, then any specialization of $f$ is also diagonal
and hence not good. In particular, Proposition~\ref{prop:diag}
does not generalize to this setting. On the other hand, since $k$
is semi-perfect, diagonal forms of degree $p$ have low strength, so this does not
contradict Proposition~\ref{prop:good}.

However, a slight variation of the strategy does work. Work in the setting of Proposition~\ref{prop:good}. We construct a plane $M$ in $V$ that is linearly independent and $f_j$-orthogonal to $F$ for all $i \le j \le r$, and such that the restriction of $f_i$ to $M$ has the form $xy^{d_i-1}$, where $x$ and $y$ are coordinates on $M$. Unfortunately, we cannot directly show that the $f_j$'s with $j>i$ vanish on $M$. To solve this problem, we iterate the construction to find planes $M_1, \ldots, M_n$ with similar properties. Since the $f_j$'s have a fairly simple form on $M_1 \oplus \cdots \oplus M_n$, we are able to find a plane $M$ inside of here having the properties of our original plane together with the necessary vanishing of $f_j$'s with $j>i$. We then appeal to Proposition~\ref{prop:orth-line} to find a line $L$ such that $f_i$ is good on $M \oplus L$, and various other conditions hold.

\subsection{Construction of adapted subplanes}

For the purposes of this section, we define a \defn{plane} to be a two-dimensional vector space $M$ equipped with an ordered basis $(v,w)$. Given a plane $(M,v,w)$ and $f \in \cP_d(M)$, write
\begin{displaymath}
f(xv+yw) = \sum_{e=0}^d c_e x^e y^{d-e}.
\end{displaymath}
We define $c_e(f)$ to be the coefficient $c_e$ above. Suppose $M_1,
\ldots, M_n$ are planes, with bases $(v_j,w_j)$. An \defn{adapted
subplane} in $M_1 \oplus \cdots \oplus M_n$ is a two dimensional subspace $M$ equipped with a basis $(v,w)$ such that $v \in \operatorname{span}(v_1, \ldots, v_n)$ and $w \in \operatorname{span}(w_1, \ldots, w_n)$. Suppose we are in this situation, and write $v=\sum_{j=1}^n \alpha_j v_j$ and $w=\sum_{j=1}^n \beta_j w_j$. Let $f \in \cP_d(M_1 \oplus \cdots \oplus M_n)$ be such that $M_1, \ldots, M_n$ are $f$-orthogonal. Then
\begin{displaymath}
c_e(f \vert_M) = \sum_{j=1}^n c_e(f \vert_{M_j}) \alpha_j^e \beta_j^{d-e},
\end{displaymath}
In particular, we see that if $c_e(f \vert_{M_j})$ vanishes for each $j$ then so does $c_e(f \vert_M)$. We will require the following result:

\newC{c:plane}
\begin{proposition} \label{prop:plane}
Suppose $\Sigma^*(d)$ holds. Let $V=M_1 \oplus \cdots \oplus M_n$ be a direct sum of planes, and let $f_1, \ldots, f_r \in \cP_d(V)$. Suppose that
\begin{enumerate}
\item $M_1, \ldots, M_n$ are $f_i$-orthogonal for each $1 \le i \le
r$,
\item $c_1(f_1 \vert_{M_j}) \ne 0$ for $1 \le j \le n$, and
\item $c_1(f_i \vert_{M_j})=0$ for $2 \le i \le r$ and $1 \le j \le n$.
\end{enumerate}
If $n>\C{c:plane}(d,r)$, then there exists an adapted subplane $M$ of $V$ such that $c_e(f_1 \vert_M)=\delta_{e,1}$ and $f_i \vert_M=0$ for $2 \le i \le r$.
\end{proposition}

Note that the conclusion implies that $f_1 \vert_M=xy^{d-1}$, where
$x$ and $y$ are the coordinates on $M$. We require two lemmas before giving the proof. The first lemma is really the key result:

\newC{c:atomic-2}
\begin{lemma} \label{lem:atomic-2}
Suppose $\Sigma^*(d)$ holds and fix $e \in \{0,\ldots,d\} \setminus
\{1\}$. Let $a_1, \ldots, a_n \in k$ be arbitrary and let $b_1, \ldots,
b_n \in k$ be nonzero. For $x,y \in k^n$, put
\begin{displaymath}
f(x,y) = a_1 x_1^e y_1^{d-e} + \cdots + a_n x_n^e y_n^{d-e}, \qquad
g(x,y) = b_1 x_1 y_1^{d-1} + \cdots + b_n x_n y_n^{d-1}.
\end{displaymath}
Assuming $n>\C{c:atomic-2}(d)$, we can find $x$ and $y$ such that
$f(x,y)=0$ and $g(x,y) \ne 0$.
\end{lemma}

\begin{proof}
If $a_i=0$ for some $i$, then we can take $x=y$ to be the $i$th
standard basis vector of $k^n$. Thus, in what follows, we assume that the
$a_i$ are nonzero. We write $f_y$ (resp.\ $f^x$) when we fix a vector
$y$ and think of $f(x,y)$ as a function of $x$ (resp.\ {\em vice
versa}), and similarly with $g$. We break the proof up into several different cases.

\textit{Case 1: $e=0$.} The defining property of Brauer fields shows
that if $n>N_k(d)$, a non-zero $y$ exists with $f(x,y)=0$; note that $f$ does not depend on $x$ in this case. Now simply pick $x$ such that $g(x,y) \ne 0$.

\textit{Case 2: $e=d$.} This is similar to Case~1, and again it suffices
that $n>N_k(d)$. 

\textit{Case 3: $p \nmid e$ and $e \ne 0,d$.} Let $y=(1,\ldots,1)$.
Then $f_y$ has strength $\geq n/2$ by Proposition~\ref{prop:diagstrength}
and the linear form $g_y$ is non-vanishing on $Z(f_y)$. By $\Sigma((e))$, we can thus
find $x$ as required if $n > 2 \cdot \C{c:mainthm}((e))$.

\textit{Case 4: $p \nmid d-e$ and $e \ne 0,d$.} Choose $x$ with all
components non-zero such that $g^x$ is not identically zero on $Z(f^x)$;
we prove below that such an $x$ exists. The argument in Case 3 shows that
if $n> 2 \cdot \C{c:mainthm}((d-e))$, then we can find a non-zero $y$
with $f^x(y)=0$ and $g^x(y) \ne 0$.

We now prove the existence of $x$. Choose $x_1,x_2 \in k^{\times}$
such that some root $(\zeta_1,\zeta_2) \in \ol{k}^2$ of $a_1x_1^e y_1^{d-e}+a_2x_2^e
y_2^{d-e}$ is not a root of $b_1x_1 y_1^{d-1}+b_2x_2y_2^{d-1}$; we
are thinking of these as binary forms in $y_1$ and $y_2$. One sees
that such $(x_1,x_2)$ exists by explicitly computing roots. Put
$x_3=\cdots=x_n=1$. Since $(\zeta_1,\zeta_2,0,\ldots,0)$ is a
$\ol{k}$-point of $Z(f^x)$ at which $g^x$ is non-zero, we see that this
$x$ works. For this we need $n > 1$. 

\textit{Case 5: $p \mid d$ and $e \ne 0,d$.} Let $x=(1, \ldots, 1)$.
Since $p \nmid d-1$, it follows that $g^x$ has strength $\geq n/2$
(Proposition~\ref{prop:diagstrength}). Proposition~\ref{prop:lowdeg}
with $g=g^x$ now shows that if $n > 2 \cdot \C{c:lowdeg}((d-e),d-1)$ we can find $y$ such that $g^x(y) \ne 0$ and $f^x(y)=0$.

Note that we are necessarily in one of the above cases, as if $p$
divides $e$ and $d-e$ then it also divides $d$. Summarizing, we see
that we can take $\C{c:atomic-2}(d)$ to be equal to the maximum of ($1$
and)
$N_k(d),2\cdot\C{c:mainthm}((e)),2\cdot\C{c:mainthm}((d-e)),2\cdot\C{c:lowdeg}((d-e),d-1)$
over all $e=2,\ldots,d-1$.
\end{proof}

We now recast the above lemma in terms of specializations of polynomials.

\begin{lemma} \label{lem:atomic-3}
Suppose $\Sigma^*(d)$ holds and fix $e \in \{0,\ldots,d\} \setminus
\{1\}$. Let $V=M_1 \oplus \cdots \oplus M_n$ be a sum of planes. Let
$f,g \in \cP_d(V)$ be such that $M_1, \ldots, M_n$ are $f$- and
$g$-orthogonal, and $c_1(g|_{M_j}) \ne 0$ for all $j$. Assuming $n>\C{c:atomic-2}(d)$,
there exists an adapted subplane $M$ of $V$ such that $c_e(f
\vert_M)=0$ and $c_1(g \vert_M) \ne 0$.
\end{lemma}

\begin{proof}
Let $(v_j,w_j)$ be the given basis of $M_j$, and put $a_j=c_e(f
\vert_{M_j})$ and $b_j=c_1(g \vert_{M_j})$. We define $M$ to be the
span of $(v,w)$ where $v=\sum_{j=1}^n x_j v_j$ and $w=\sum_{j=1}^n y_j
w_j$, and $x,y \in k^n$ are to be determined. As we have already
seen, $c_e(f \vert_M)$ and  $c_1(g \vert_M)$ are exactly the
quantities $f$ and $g$ appearing in Lemma~\ref{lem:atomic-2}. Thus, by
that lemma, we can find $x$ and $y$ such that $c_e(f \vert_M) = 0$ and
$c_1(g \vert_M) \ne 0$. Note that $x$ and $y$ are necessarily non-zero (since $F \ne 0$), and so $v$ and $w$ are linearly independent vectors. This completes the proof.
\end{proof}

\begin{proof}[Proof of Proposition~\ref{prop:plane}]
We claim that the result holds with $\C{c:plane}(d,r) =
(\C{c:atomic-2}(d)+1)^{rd}$. Let $f_1, \ldots, f_r$ and $V=M_1 \oplus \cdots
\oplus M_n$ be given as in the proposition. Each $f_i$ has $d$
coefficients that we wish to make vanish, namely the $c_e$ for $e \in
\{0,\ldots,d\} \setminus \{1\}$. Thus there are $rd$ coefficients total
that we wish to make vanish. We will accomplish this in $rd$ steps,
where we make one coefficient vanish in each step.

Put $M_i^0=M_i$. Break the list $M^0_1, \ldots, M^0_n$ up into
$(\C{c:atomic-2}+1)^{rd-1}$ blocks of size $>\C{c:atomic-2}$. Applying
Lemma~\ref{lem:atomic-3} to the $i$th block, we can find an adapted
subplane $M^1_i$ of the direct sum of planes in this block such that
$c_1(f_1 \vert_{M^1_i})$ is non-zero and $c_e(f_{\ell} \vert_{M^1_i})=0$,
where $e \in \{0,\ldots,d\} \setminus \{1\}$ and $1 \le \ell \le
r$ are of our choosing (we choose the same $e$ and $\ell$ in each
block). Thus after this first step, we have planes $M^1_i$ for $1 \le i
\le (\C{c:atomic-2}+1)^{rd-1}$ and our chosen coefficient of $f_{\ell}$
vanishes on each plane. We now continue in the same manner with the
restrictions of the $f$s to the direct sum of the blocks $M_i^1$, moving onto a
different coefficient (of either the same $f_\ell$ or a different one). A key point
in this procedure is that one we have made some coefficient, say $c_e$
of $f_{\ell}$, vanish, it remains vanishing in each subsequent step. 
\end{proof}

\subsection{Derivatives}

Given $f \in \cP_d(V)$, we let $Df \in \cP_d(V \oplus V)$ be the bi-degree $(1, d-1)$ piece of the polynomial $(v,w) \mapsto f(v+w)$.

\newC{c:Df}
\begin{proposition} \label{prop:Df}
The strength of $Df$ can be lower-bounded by a function of $d$ and $\str(f)$. 
\end{proposition}

\begin{proof}
Let $I \subseteq \cP(V)$ and $J \subseteq \cP(V \oplus V)$ be the
ideals generated by the first-order partial derivaties of $f$ and $Df$,
respectively.  That means that $I$ is generated by the directional
derivatives $w \mapsto \frac{\partial f}{\partial v}(w)$ as $v$ ranges
over $V$. The right-hand side is the coefficient of $\epsilon^2$ in
$f(\epsilon v + w)$, which is also the value of the directional derivative
$\frac{\partial Df}{\partial (v,0)}$ at any point $(\tilde{v},w)$ with
$\tilde{v} \in V$ arbitrary. This implies that the projection of $Z(J)$
into the second factor is contained in $Z(I)$.

Now by Theorem~\ref{thm:sing} the strength of $Df$ can be lower-bounded
in terms of ($d$ and) the codimension of $Z(J)$ in $V \oplus V$.  By the above
this codimension is at least that of $Z(I)$ in $V$, which again by
Theorem~\ref{thm:sing} can be lower-bounded in terms of $\str(f)$.
\end{proof}

\begin{corollary} \label{cor:D-strength}
The strength of the tuple $(Df_1, \ldots, Df_r)$ can be lower-bounded as a
function of $\ul{d}$ and the strength of $\ul{f} \in \cP_{\ul{d}}(V)$. 
\end{corollary}

\begin{proof}
Consider a non-trivial homogeneous linear combination $\sum_{i=1}^r c_i
Df_i$. Since $D$ is linear, this equals $D(\sum_{i=1}^r c_i f_i)$, whose
strength is lower-bounded by a function of ($\ul{d}$ and) $\sum_{i=1}^r
c_i f_i$ by Proposition \ref{prop:Df}. The latter strength, in turn,
is lower-bounded by that of $\ul{f}$. 
\end{proof}

\begin{remark} \label{rmk:D2}
Define $D_2f$ to be the bi-degree $(2,d-2)$ piece of $(v,w) \mapsto f(v+w)$. It is not necessarily true that $D_2f$ has high strength when $f$ does. Indeed, suppose $p$ is odd and let $f=x_1^{p+1}+\cdots+x_n^{p+1}$ with $n \gg 0$. Then $f$ has high strength by Proposition~\ref{prop:diagstrength}, but $D_2f=0$. We thank Amichai Lampert for showing us this example.
\end{remark}

\subsection{Construction of orthogonal planes}

We now show that we can find a plane orthogonal to a given subspace, and having some other properties.

\newC{c:orth-plane}
\begin{proposition} \label{prop:orth-plane}
Assume $\Sigma^*(\ul{d})$ holds. Let $\ul{f} \in \cP_{\ul{d}}(V)$, let $1 \le i \le r$, and let $F$ be a subspace of $V$. Suppose $\str(\ul{f})>\C{c:orth-plane}(\ul{d}, \dim{F})$. Then we can find a plane $M \subset V$ such that
\begin{enumerate}
\item $F$ and $M$ are linearly independent,
\item $F$ and $M$ are $f_j$-orthogonal for $i \le j \le r$, 
\item $f_j \vert_M=0$ for all $j$ with $d_j<d_i$,
\item $c_1(f_j \vert_M)=0$ for all $j>i$ with $d_j=d_i$, and 
\item $c_1(f_i \vert_M) \ne 0$.
\end{enumerate}
\end{proposition}

\begin{proof}
Let $W$ be a complementary space to $F$ in $V$. We aim to find $(v,w) \in W \times W$ satisfying the following conditions:
\begin{enumerate}[(i)]
\item $v$ and $w$ are linearly independent, 
\item $\operatorname{span}(v,w)$ and $F$ are $f_j$-orthogonal for $i
\le j \le r$,
\item $v$ and $w$ are $f_j$-orthogonal for all $j$ with $d_j<d_i$,
\item $f_j(v)=f_j(w)=0$ for all $j$ with $d_j<d_i$,
\item $(Df_j)(v,w) = 0 $ for all $j>i$ with $d_j=d_i$, and
\item $(Df_i)(v,w) \ne 0$.
\end{enumerate}
We can then take $M=\operatorname{span}(v,w)$. Note that (iii) and (iv) together are equivalent to (c), while (v) and (vi) are equivalent to (d) and (e).

Consider the closed subvariety $Z$ of $W \times W$ defined by
the equations in (ii)--(v). Let $\ul{e}$ be the multi-degree for these equations. All of these equations have degree $\le d_i$, and the only equations of degree $d_i$ are those in
(v), say $Df_{i+1}, \ldots, Df_{\ell}$; we thus see that $\ul{e}<\ul{d}$. The polynomials $Df_i, \ldots, Df_{\ell}$ have high strength by
Corollary~\ref{cor:D-strength}. We now apply Proposition~\ref{prop:lowdeg} (with $g=Df_i$). This yields a closed subvariety $Z'$ of $Z$ of small codimension such that $Z'(k)$ is dense in $Z'$ and $Df_i$ is non-vanishing on $Z'$. Thus there is a $k$-point $(v,w)$ of $Z'$
at which $Df_i$ is non-vanishing. This pair $(v,w)$ satisfies (ii)--(vi).

It remains to show that we can choose $v$ and $w$
linearly independent. Let $Y \subset W \times W$ be the set of pairs
$(v,w)$ that are linearly dependent. This set is Zariski closed (it
is a determinantal variety), and has dimension $1+\dim(W)$. Since the
codimension of $Z'$ in $W \times W$ is bounded from above, it follows
that the open set $Z' \setminus Y$ of $Z'$ is non-empty. The $k$-point
from the previous paragraph at which $Df_i$ is non-vanishing can be
chosen in this open set. This completes the proof.
\end{proof}

\begin{remark}
It does not seem possible to directly obtain $f_j \vert_M=0$ for all $i<j \le r$ in the above proof. Indeed, this is equivalent to $d_j+1$ equations of degree $d_j$. If we added these equations to our list, then we could potentially have $\ul{e} \ge \ul{d}$, putting us outside of the range of our inductive hypothesis. We would not be able to apply Proposition~\ref{prop:lowdeg}.
\end{remark}

\subsection{Completion of proof}

We now prove Proposition~\ref{prop:good} in characteristic~$p$. We restate the result for convenience:

\begin{proposition} \label{prop:goodp}
Assume $\Sigma^*(\ul{d})$ holds and $k$ has characteristic~$p$. Let $\ul{f} \in \cP_{\ul{d}}(V)$, let $1
\le i \le r$, and let $F$ be a subspace of $V$. If $\str(\ul{f})>\C{c:good}(\ul{d}, \dim{F})$, then there is a three dimensional subspace $E$ of $V$ such that:
\begin{enumerate}
\item $E$ and $F$ are linearly independent,
\item $E$ and $F$ are $f_j$-orthogonal for $i \le j \le r$,
\item $f_j$ vanishes on $E$ for $i<j \le r$, and
\item the restriction of $f_i$ to $E$ is good.
\end{enumerate}
\end{proposition}

\begin{proof}
For a sufficiently large $n$ to be chosen below, we first
find planes $M_1,\ldots,M_n$ such that:
\begin{enumerate}[(i)]
\item $M_1,\ldots,M_n,F$ are linearly independent,
\item $M_1,\ldots,M_n,F$ are $f_j$-orthogonal for $i \leq j \leq r$,
\item $f_j|_{M_\ell}=0$ for all $\ell$ and all $j$ with $d_j<d_i$,
\item $c_1(f_j|_{M_\ell})=0$ for all $j>i$ with $d_j=d_i$, and
\item $c_1(f_i|_{M_\ell}) \neq 0$. 
\end{enumerate}
Here $M_1$ is constructed by applying Proposition~\ref{prop:orth-plane}
to $F$, $M_2$ by applying the same proposisition to $M_1 \oplus F$, and so on. For this we need that
$\str(\ul{f})>\C{c:orth-plane}(\ul{d},\dim(F)+2(n-1))$ if we assume, as
we may, that $\C{c:orth-plane}$ is non-decreasing in its second argument.

We now apply Proposition~\ref{prop:plane} to the polynomials $f_j$ for
which $j \geq i$ and $d_j=d_i$; let $r'$ be their number.  We then
find an adapted subplane $M$ of $M_1 \oplus \cdots \oplus M_n$
such that $c_e(f_i|_M)=\delta_{e,1}$ and $f_j|_M=0$ for all $j >i$
with $d_j=d_i$. This exists provided that $n$ is chosen greater than
$\C{c:plane}(d_i,r')$.  Since $f_j|_{M_\ell}=0$ for all $j$ with $d_j<d_i$
and all $\ell$, we now have $f_j|_M=0$ for all $j>i$, while $f_i|_M$
has the form $xy^{d_i-1}$ in the coordinates on $M$. Furthermore, we
$M,F$ are $f_j$-orthogonal for all $j \geq i$.

Next we apply Proposition~\ref{prop:orth-line} to $M \oplus F$
to find a line $L$ in $V$ such that:
\begin{enumerate}[(i)]
\item $L$ and $M \oplus F$ are linearly independent,
\item $L$ and $M \oplus F$ are $f_j$-orthogonal for $i \leq j \leq r$, 
\item $f_j|_L=0$ for $i<j \leq r$, and 
\item $f_i|_L \neq 0$. 
\end{enumerate}
For this we need that $\str(\ul{f})>\C{c:orth-line}(\ul{d},2+\dim(F))$. 

Finally, take $E=M \oplus L$. The restriction of $f_i$ to $E$ is good,
and all $f_j$ with $j>i$ vanish identically on $E$. By construction, $E$
and $F$ are linearly independent and $f_j$-orthogonal for all $j \geq i$.
\end{proof}

\begin{remark}
The proof in this section also works in characteristic zero.  However,
we have decided to first explain the simpler proof in characteristic
zero in \S \ref{s:normal0} using diagonal forms and only then give
the more technical proof for the general case using adapted planes. 
\end{remark}

\end{document}